\documentclass[a4paper,12pt]{amsart}
\usepackage{amsfonts}
\usepackage{amsthm}
\usepackage{amssymb}
\usepackage{amsmath}
\usepackage{enumerate}
\usepackage{url}
\usepackage{color}
\usepackage{comment}
\usepackage{a4wide}

\usepackage{bbm}

\newtheorem{lemma}{Lemma}
\newtheorem{thm}{Theorem}

\theoremstyle{definition}

\theoremstyle{remark}
\newtheorem*{rmk}{Remark}
\newtheorem*{rmks}{Remarks}

\newcommand{\ve}{\varepsilon}

\allowdisplaybreaks

\begin{document}
\title{Difference sets and the metric theory of small gaps}
\author[C.\ Aistleitner, D.\ El-Baz, M.\ Munsch]{Christoph Aistleitner, Daniel El-Baz, Marc Munsch}


\newcommand{\mods}[1]{\,(\mathrm{mod}\,{#1})}

\begin{abstract}
Let $(a_n)_{n \geq 1}$ be a sequence of distinct positive integers. In a recent paper Rudnick established asymptotic upper bounds for the mini\-mal gaps of $\{a_n \alpha ~\textup{mod}~ 1,  ~1 \leq n \leq N\}$ as $N \to \infty$, valid for Lebesgue-almost all $\alpha$ and formulated in terms of the additive energy of $\{a_1, \dots, a_N\}$. In the present paper we argue that the metric theory of minimal gaps of such sequences is not controlled by the additive energy, but rather by the cardinality of the difference set of $\{a_1, \dots, a_N\}$.  We establish a (complicated) sharp convergence/divergence test for the typical asymptotic order of the minimal gap, and prove (slightly weaker) general upper and lower bounds which allow for a direct application. A major input for these results comes from the recent proof of the Duffin--Schaeffer conjecture by Koukoulopoulos and Maynard. We show that our methods give very precise results for slowly growing sequences whose difference set has relatively high density, such as the primes or the squares. Furthermore, we improve a metric result of Blomer, Bourgain, Rudnick and Radziwi{\l}{\l} on the order of the minimal gap in the eigenvalue spectrum of a rectangular billiard. 
\end{abstract}

\maketitle

\section{Introduction and statement of results}

Let $(a_n)_{n \geq 1}$ be a sequence of distinct positive integers. Many authors have studied the distribution of sequences of the form $(a_n \alpha)_{n \geq 1}$ mod 1, either in the setup when $\alpha$ is a fixed real number (usually irrational) or in the metric setup where $\alpha$ is randomly chosen from the unit interval, equipped with Borel sets and Lebesgue measure. Particular attention has been given to the case when $(a_n)_{n \geq 1}$ is exponentially increasing \cite{RZ1999, RZ2002}, is generated by the values at integers of an (integer-valued) polynomial \cite{rs, rsz}, or is the sequence of primes \cite{mato,Walker}. The ``global'' distributional behavior of such sequences (mod 1) is described by uniform distribution theory and quantified by discrepancy theory; see for example \cite{dts,kn} for a general exposition. The ``local'' distribution properties can be described in terms of pair correlation, higher correlations, and neighbor gap statistics. This viewpoint has one of its motivations in quantum physics; see \cite{mark} for a survey in mathematical language. For more context and examples we refer to \cite{rs,rsz}. In the present paper we are concerned with the minimal gap statistic for sequences $(a_n \alpha)_{n \geq 1}$. For $\alpha \in [0,1]$, we define
\begin{equation} \label{delta_def}
\delta_{\min}^{\alpha} (N) = \min \Big\{ \| (a_m - a_n)  \alpha\|: ~1 \leq m, n \leq N,~m \neq n \Big\},
\end{equation}
where $\| \cdot \|$ denotes the distance on the torus.  Let $A_N$ denote the finite truncations $(a_1, \dots, a_N)$ of the sequence $(a_n)_{n \geq 1}$. Let $E_N$ denote the additive energy of $A_N$, that is, the number of solutions $(n_1,n_2,n_3,n_4)$ with $n_i \leq N,~1 \leq i \leq 4,$ to the equation $a_{n_1} - a_{n_2} = a_{n_3} - a_{n_4}$. Pursuing an idea that originated in the context of pair correlation problems (see \cite{all,bw}), Rudnick \cite{rud} proved the following. Let $\ve > 0$. Then for almost all $\alpha$  and all sufficiently large $N$
\begin{equation} \label{rud}
\delta_{\min}^{\alpha} (N) \leq \frac{E_N}{N^{4 - \ve}},
\end{equation}
while in the other direction, 
\begin{equation} \label{rud_2}
\delta_{\min}^{\alpha} (N) \geq \frac{1}{N^{2 + \ve}}
\end{equation}
for almost all $\alpha$ and all $N \geq N_0(\alpha)$. It is easy to see that for every sequence $E_N$ lies in the range $N^2 \leq E_N \leq N^3$. Thus \eqref{rud} gives a non-trivial result whenever $E_N \ll N^{3 - \ve'}$ for some $\ve' > 0$, i.e.\ whenever there is a power saving in the bound for the additive energy. If the additive energy is of smallest possible order, say $E_N \ll N^{2 + o(1)}$, then \eqref{rud} reads $\delta_{\min}^{\alpha} (N) \ll \frac{1}{N^{2 - \ve}}$, which in view of \eqref{rud_2} is optimal up to terms of order $N^{o(1)}$. As noted in \cite{rud}, this yields strong results for sequences where the additive energy is small, for example in the case when $(a_n)_{n \geq 1}$ is exponentially increasing or when $a_n = n^2,~n \geq 1$.

In the setting of the metric theory of pair correlations, the additive energy is a very appropriate tool to establish pseudo-random (so-called Poissonian) behavior, even if it is known that the additive energy alone is not sufficient for a full understanding of that theory \cite{alt}. By contrast, we argue in this paper that the metric theory of minimal gaps is better characterized in terms of the cardinality of the difference set, rather than by the additive energy. Heuristically this is quite reasonable: while in the pair correlation setting \emph{all} pairs $(m,n)$ with a certain difference $a_m - a_n$ contribute to the overall statistics, and many solutions of $a_{n_1} - a_{n_2} = a_{n_3} - a_{n_4}$ lead to an overshoot of small gaps, in the setup of minimal gaps every specific difference $a_m - a_n$ can contribute only \emph{once} to creating a particularly small gap $(a_m - a_n) \alpha$ mod 1. We write 
$$
A_N - A_N = \left\{a - b:~a,b \in A_N \right\} \qquad \text{and} \qquad C_N = \# (A_N - A_N).
$$ 
The claim is then that the order of $\delta_{\min}^{\alpha} (N)$ for almost all $\alpha$ is essentially controlled by the asymptotic order of $C_N$. Note that $C_N$ and $E_N$ are of course related. If $\textup{rep}_N(u)$ denotes the number of possible representations of an integer $u$ in the form $u=a_m - a_n$, subject to $m,n \leq N$, then
$$
C_N = \sum_u \mathbbm{1} (\textup{rep}_N (u) > 0) \qquad \text{and} \qquad E_N = \sum_u \textup{rep}_N(u)^2, 
$$
so that 
\begin{equation} \label{C_N_E_N} 
N^4/C_N \leq E_N \leq N^2 C_N 
\end{equation}
by Cauchy--Schwarz.  For a ``random'' sequence $(a_n)_{n \geq 1}$ one might expect an average number $\textup{rep}_N(u) \approx N^2 / C_N$ of representations for every $u \in A_N - A_N$, so that $E_N \approx N^4 / C_N$.
Theorem \ref{th1} shows that the typical order of $\delta_{\min}^{\alpha} (N)$ can be controlled very precisely in terms of the cardinality of the difference set $C_N$.\footnote{In the formulation of the theorem and throughout the paper we write $\log_2 x = \log \log x$, ~$\log_3 x = \log \log \log x$, etc. We will always read $\log x$ as $\max\{1, \log x\}$. Thus $1/\log n$ is well-defined for all integers $n \geq 1$. Furthermore, $\log_2 n$, $\log_3 n,$ and so on are positive and well-defined for all integers $n \geq 1$.}
\begin{thm} \label{th1}
Let $(a_n)_{n_ \geq 1}$ be a strictly increasing sequence of positive integers. Let $\ve > 0$. Then for almost all $\alpha \in [0,1]$ we have
\begin{equation} \label{th1_1}
\delta_{\min}^{\alpha} (N) \leq \frac{1}{C_N}\frac{\log_2 a_N}{\log N \log_2 N} \qquad \text{for infinitely many $N$},
\end{equation}
as well as
\begin{equation} \label{th1_2}
\delta_{\min}^{\alpha} (N) \geq \frac{1}{C_N \log N(\log_2 N)^{1+\ve}} \qquad \text{for all sufficiently large $N$.}
\end{equation}
The conclusion \eqref{th1_1} is also true for almost all $\alpha$ if the factor $\frac{\log_2 a_N}{\log N \log_2 N}$ is replaced by 1, thus yielding a result that depends only on $C_N$ and is independent of the actual size of the elements of $(a_n)_{n \geq 1}$.
\end{thm}

\begin{rmks}
\begin{enumerate}[(i)]
\item The effect of the cardinality of the difference set on the asymptotic order of the minimal gap has already been observed in Regavim's recent paper \cite{reg}, where equation \eqref{th1_2} appears in Section 14. Trivially $N \leq C_N  \leq N^2$, so \eqref{th1_2} always improves upon \eqref{rud_2}. Since the proof of \eqref{th1_2} can be given in a few lines, we include it below for the sake of completeness. 
\item By \eqref{C_N_E_N}, equation \eqref{th1_1} always improves upon \eqref{rud}.
Note, however, that the two statements are of a different nature. While \eqref{rud} holds for all sufficiently large $N$, \eqref{th1_1} holds for infinitely many $N$. The purpose of \eqref{th1_1} is to show that \eqref{th1_2} cannot be significantly improved, thus fixing the \emph{lower} endpoint of the range within which $\delta_{\min}^{\alpha}$ is contained for all except finitely many $N$. By contrast, \eqref{rud} is about determining the \emph{upper} endpoint of the range within which $\delta_{\min}^{\alpha}$ is contained for all except finitely many $N$. At the end of the proof of Theorem \ref{th1} we will comment on the difficulties to pass from a result for ``infinitely many $N$'' to one for ``all except finitely many $N$'' (but see Theorem \ref{t:allN} for a result of the latter type). We do believe, however, that the upper endpoint of the window within which $\delta_{\min}^{\alpha}$ is contained for all sufficiently large N (for almost all $\alpha$) is very close to the lower endpoint of that window, as provided by Theorem \ref{th1}. Using an argument similar to the one leading to \eqref{th1_2}, we can prove that for almost all $\alpha$
\begin{equation*}
\delta_{\min}^{\alpha} (N) \geq \frac{f(N)}{C_N} \qquad \text{for infinitely many $N$},
\end{equation*}
where $f$ is an arbitrary function tending to zero slowly as $N \to \infty$ (compare with the proof of the second part of Theorem \ref{th2} below). We believe that this is essentially optimal (say up to a factor $(\log N)^\ve$), and that for almost all $\alpha$ and every fixed $\ve>0$
\begin{equation} \label{conj_up}
\delta_{\min}^{\alpha} (N) \leq \frac{(\log N)^\ve}{C_N} \qquad \text{for all except finitely many $N$}.
\end{equation}

\item For rapidly (i.e.\ super-exponentially) growing sequences $(a_n)_{n \geq 1}$ we use \eqref{th1_1} without the factor $\frac{\log_2 a_N}{\log N \log_2 N}$ so that it becomes
$$
\delta_{\min}^{\alpha} (N) \leq \frac{1}{C_N} \qquad \text{for infinitely many $N$},
$$
which means that there is a discrepancy of logarithmic order in comparison with the lower bounds \eqref{th1_2}. This gap probably cannot be fully closed without taking fine arithmetic effects into account. This gap between \eqref{th1_1} and \eqref{th1_2} reflects the fact that Khintchine's theorem in metric Diophantine approximation generally fails without a monotonicity assumption, or without allowing a fine arithmetic criterion such as in the formulation of Catlin's conjecture, which will lead to our more precise Theorem \ref{th_cat} below.

\item For slowly (say polynomially) growing sequences $(a_N)_{n \geq 1}$ the gap between upper and lower bounds in Theorem \ref{th1} is only of order $\log_2 a_N \ll \log_2 N$. However we will show that, for sequences for which the difference set has a high relative density (within the maximal possible difference set $\{1, \dots, a_N\}-\{1, \dots, a_N\}$), the estimates from Theorem \ref{th1} can be strengthened even further to give extremely precise results. We illustrate this by considering as examples the sequences of primes and squares (Theorems \ref{examples} and \ref{th_quad} below, respectively). 

\item The conclusion of Theorem \ref{th1} is close to Rudnick's upper bound \eqref{rud} in cases where $E_N \approx N^2$ and thus $C_N \approx N^2$ (only improving the error $N^\varepsilon$ to errors of logarithmic order). Rudnick's result also meets with that of Theorem \ref{th1} when $E_N \approx N^3$ and $C_N \approx N$. However, in the intermediate range when $E_N$ is around $N^{\delta}$ for some $\delta \in (2,3)$, the estimate \eqref{rud} is significantly weaker than \eqref{th1_1}, except in the ``random'' case where all possible differences $a_m -a_n$ have a similar number of representations so that $E_N \approx N^4 / C_N$. Examples of such sequences include the Piatetski-Shapiro integers as well as the sequence of $(\log x)^A$ smooth numbers $n\leq x$ where we expect $C_N \approx N^{\beta_A}$ for some $\beta_A$ in the open interval $(1,2)$; see \cite{Banks} for more information.  
\item From the viewpoint of metric number theory, an upper bound such as \eqref{th1_1} is usually much more difficult to establish than a lower bound such as \eqref{th1_2}. This is because the upper bound \eqref{th1_1} uses the second rather than the first Borel--Cantelli lemma, which only holds under some additional ``stochastic independence'' requirement. We will comment on this in more detail during the proof of Theorem \ref{th1}. 
\end{enumerate}
\end{rmks}

To complement the discussion in Remark (ii) above, we note that it is possible to adapt Rudnick's method from \cite{rud} in order to obtain an upper bound valid for all but finitely many $N$ expressed in terms of $C_N$.

\begin{thm} \label{t:allN}
For every $\ve > 0$ and almost all $\alpha$, 
\[ \delta_{\min}^\alpha(N) \le \frac{N^\ve}{C_N} \qquad \text{for all except finitely many } N.
\]
\end{thm} 
\begin{rmk} 
The above bound should be compared with our guess \eqref{conj_up}.
See also the second remark at the end of Section \ref{s:pf_squares} for a case in which such an approach gives a significant improvement, namely when $(a_n)_{n \ge 1} = (n^2)_{n \ge 1}$.
\end{rmk}

Theorem \ref{th_cat} below shows that the problem of determining the typical asymptotic order of the minimal gap of $(a_n \alpha)_{n \geq 1}$ can be completely solved (at least in principle) solely by considering the difference sets $A_N - A_N$. For a set $A$ we denote by $A^+$ the set of all positive elements $A$. For $k \in \bigcup_{N \geq 1} (A_N - A_N)^+$, we define 
\begin{equation} \label{nkdef}
\mathcal{N}(k) = \min \{M \geq 1:~k \in (A_M - A_M)^+\}.
\end{equation} 
In words, $\mathcal{N}(k)$ is the time at which a certain positive integer $k$ first occurs in a difference set $A_M - A_M$ (and can contribute to producing a small gap $\|k \alpha\|$). For $k \not\in \bigcup_{N \geq 1} (A_N - A_N)^+$, we set $\mathcal{N}(k) = \infty$. Throughout the paper, $\varphi$ denotes Euler's totient function, and the supremum of the empty set is understood to be zero.

\begin{thm} \label{th_cat}
Let $(a_n)_{n_ \geq 1}$ be a strictly increasing sequence of positive integers. Let $(\eta(N))_{N \ge 1}$ be a sequence of non-negative reals. Let $\mathcal{A}$ denote the set of those $\alpha \in [0,1]$ for which $\delta_{\min}^{\alpha}(N) \leq \eta(N)$ holds for infinitely many $N$. Then we have $\lambda(\mathcal{A})= 0$ or  $\lambda(\mathcal{A})= 1$, according to whether the series
\begin{equation} \label{con_div}
\sum_{k=1}^\infty \varphi(k) \sup_{b \geq 1} \left\{\frac{\sup_{\ell \geq \mathcal{N}(bk)} \eta(\ell)}{bk} \right\}
\end{equation}
is convergent or divergent, respectively.
\end{thm}

\begin{rmks}
\begin{enumerate}[(i)]
\item Note that under the (reasonable) assumption that $\eta$ is decreasing the convergence/divergence criterion in Theorem \ref{th_cat} takes the simpler form
$$
\sum_{k=1}^\infty \varphi(k) \sup_{b \geq 1} \left\{ \frac{\eta(\mathcal{N}(bk))}{bk} \right\} \qquad \text{is convergent or divergent, respectively.}
$$

\item From a purely formal perspective, Theorem \ref{th_cat} is a complete solution to the problem of the asymptotic order of $\delta_{\min}^{\alpha}(N)$ for almost all $\alpha$. However, the practical value of Theorem \ref{th_cat} is of course limited since the convergence/divergence criterion involving the series \eqref{con_div} is more or less impossible to check in reality. Note that an evaluation of \eqref{con_div} would require a fully detailed understanding of the arithmetic structure of the difference sets $A_N - A_N$, something that is out of reach even in very ``simple'' cases such as $a_n = n^2, ~n \geq 1$.
\end{enumerate}
\end{rmks}

The following result (for the primes) is a strengthening of a result due to Rudnick \cite[Section $3$]{rud}, which only contained a lower bound of the correct order of magnitude. We obtain a precise convergence/divergence criterion, which does not contain any errors whatsoever. The assumptions on $\eta$ (monotonicity, regularly varying) are only to simplify the statement of the theorem.
\begin{thm}\label{examples}
Let $a_n=p_n$ be the $n$-th prime. Let $(\eta(N))_{N \geq 1}$ be a strictly decreasing sequence of non-negative reals such that $\eta(2N)/\eta(N) \gg 1$. Let $\mathcal{A}$ denote the set of those $\alpha \in [0,1]$ for which 
$$
\delta_{\min}^{\alpha}(N) \leq \eta(N \log N)
$$ 
holds for infinitely many $N$. Then we have $\lambda(\mathcal{A})= 0$ or  $\lambda(\mathcal{A})= 1$, according to whether the series
\begin{equation*} \label{con_div_ex}
\sum_{N=1}^\infty \eta(N)
\end{equation*}
is convergent or divergent, respectively.
\end{thm} 
Theorem \ref{examples} implies for example that for the sequence of primes, for almost all $\alpha$
\begin{equation*}
\delta_{\min}^{\alpha} (N) \leq \frac{1}{N(\log N)^2(\log_2 N)(\log_3 N)} \qquad \text{for infinitely many $N$},
\end{equation*}
while
\begin{equation*}
\delta_{\min}^{\alpha} (N) \geq \frac{1}{N (\log N)^{2}(\log_2 N)(\log_3 N)^{1+\ve}} \qquad \text{for all sufficiently large $N$.}
\end{equation*}

The next example concerns the sequence of squares (cf.\ Regavim's result in \cite[Section 11]{reg}). Using a combinatorial input on the cardinality of the difference set we obtain a result where the size of the error is only a power of $\log_3 N$. In the statement of the next theorem, $c = 1-(1+\log_2 2)/(\log 2) \approx 0.086$ is the constant from the answer to the Erd\H os multiplication table problem; see \cite{ford,ten}.

\begin{thm}\label{th_quad}
 Let $a_n = n^2$, $n\geq 1$, be the sequence of squares. For every $\ve > 0$ and almost all $\alpha \in [0,1]$, we have \begin{equation*} \delta_{\min}^{\alpha} (N) \leq \frac{(\log N)^{c-1} (\log_2 N)^{1/2}}{ N^2}  \qquad \text{for infinitely many $N$,} \end{equation*} as well as
\begin{equation*}
\delta_{\min}^{\alpha} (N) \geq  \frac{(\log N)^{c-1} (\log_2 N)^{1/2}}{ N^2 (\log_3 N)^{1 + \ve}}   \qquad \text{for all sufficiently large $N$.}
\end{equation*}
\end{thm}

\begin{rmk}
We note that several interesting metric results on minimal gaps have recently been obtained by Regavim \cite{reg} in the setting when $(a_n)_{n \geq 1}$ is a sequence of reals rather than integers. The case of real-valued sequences is technically much more complicated, and some extra assumptions are typically necessary, ensuring for example that the minimal spacings of $(a_n)_{n \geq 1}$ itself are not too small. We will not pursue this direction here, but it would be interesting to know to what extent the reasoning from the present paper could be transferred to the real-valued setup.
\end{rmk}

Now we switch to a different setting. Consider the set of all numbers of the form $\alpha m^2 + n^2$, where $m,n \geq 1$ and $\alpha > 0$. These numbers represent the energy spectrum of a rectangular billiard. The case $\alpha \in \mathbb{Q}$ is special, and will not be considered in this paper. If $\alpha \not\in \mathbb{Q}$, then the spectrum is simple, and we can write $0 < \lambda_1 < \lambda_2 < \dots$ for the set $\{\alpha m^2 + n^2:~m,n \geq 1\}$, sorted in increasing order. The asymptotic order of $(\lambda_k)_{k \geq 1}$ is given by counting lattice points, which yields that
\[
\lambda_k \sim \frac{4 \sqrt{\alpha}}{\pi} k
\]
as $k \to \infty$.  The Berry--Tabor conjecture predicts on a very general level that local statistics of the energy spectrum of most integrable quantum systems should follow the Poissonian model. In the case of rectangular billiards, the conjecture is assumed to be true for irrational $\alpha$ which cannot be approximated very well by rationals. Sarnak \cite{sar} proved that the pair correlation of $(\lambda_k)_{k \geq 1}$ is Poissonian for almost all $\alpha$, and Eskin, Margulis and Mozes \cite{emm} established the corresponding result for individual values of $\alpha$ satisfying a weak Diophantine assumption. Blomer, Bourgain, Radziwi{\l}{\l} and Rudnick \cite{bbrr} studied the asymptotic order of the minimal gap statistic of that sequence, that is, the quantity
$$
\delta_{\min}^{(\alpha)} (N) = \min \{ \lambda_{k+1} - \lambda_k :~1 \leq k \leq N\}. 
$$
They proved asymptotic upper and lower bounds for $\delta_{\min}^{(\alpha)} (N)$ assuming that $\alpha$ satisfies certain Diophantine approximation properties; for example they established results for certain quadratic irrationals, which have since been extended to all positive quadratic irrationals by Carmon \cite{Carmon}, and for algebraic irrationals of higher degree. We refer to \cite{bbrr} for the precise statement of their results. In the metric setup, they proved that for every $\ve > 0$ and almost all $\alpha > 0$ 
$$
\delta_{\min}^{(\alpha)} (N) \ll N^{-1 + \ve} \qquad \text{as $N \to \infty$}
$$
and
\begin{equation}\label{e:BBRR}
\delta_{\min}^{(\alpha)} (N) \gg \frac{(\log N)^c}{N} \qquad \text{infinitely often}. 
\end{equation}
As in the statement of Theorem \ref{th_quad}, $c$ denotes the constant from the multiplication table problem. It should be noted that these metric results are \emph{not} in accordance with the behavior of the Poisson process, where a convergence/divergence criterion precisely quantifies the almost sure asymptotic order of the smallest gap; in particular, in the case of the Poisson process there are (almost surely) infinitely many $N$ for which $\delta_{\min} (N) \leq \frac{1}{N \log N}$, but only finitely many $N$ for which $\delta_{\min} (N) \leq \frac{1}{N (\log N)^{1 + \ve}}$. In the opposite direction, there are almost surely infinitely many $N$ for which $\delta_{\min} (N) \geq \frac{\log_2 N}{N}$, but only finitely many $N$ for which $\delta_{\min} (N) \geq \frac{(1+\ve)\log_2 N}{N}$. See \cite{dev} for more details. This deviation between the order of minimal gaps in the Poisson process in comparison to minimal gaps in the spectrum of rectangular billiards is directly related to the cardinality of the difference set of the squares, and to the fact that a multiplication table of size $N \times N$ does not contain $N^2$ but only $o(N^2)$ distinct entries. Thus the following theorem is very much in accordance with the ``cardinality of difference set controls the typical order of minimal gaps'' philosophy, even if from the viewpoint of Diophantine approximation the situation is now much more delicate. As it will turn out during the proofs, for this problem we will not only encounter the cardinality of a difference set that controls the number of admissible denominators in the Diophantine approximation problem (as in all the previous theorems), but now there will also be a second such phenomenon with respect to the cardinality of the admissible set of numerators. We will prove the following.

\begin{thm} \label{th2}
Let $c = 1-(1+\log_2 2)/(\log 2)$ and $\varepsilon>0$. Then for almost all $\alpha > 0$
\begin{equation} \label{th2_b}
\delta_{\min}^{(\alpha)} (N) \leq \frac{(\log N)^{2c}}{N \log N}  \qquad \text{only finitely often},
\end{equation}
as well as
\begin{equation} \label{th2_c}
\delta_{\min}^{(\alpha)} (N) \geq \frac{(\log N)^{2c}}{N} \qquad \text{infinitely often}.
\end{equation}
\end{thm}

We note that \eqref{th2_c} improves on the exponent of $\log N$ in \eqref{e:BBRR} (\cite[Theorem 1.3]{bbrr}), but the main significance of Theorem \ref{th2} is that we believe \eqref{th2_b} and \eqref{th2_c} to be optimal, up to factors of order $(\log N)^{o(1)}$. More precisely, we believe that in the setting of Theorem \ref{th2} for any $\ve>0$
\begin{equation} \label{th2_a}
\delta_{\min}^{(\alpha)} (N) \leq \frac{(\log N)^{2c+\ve}}{N \log N} \qquad \text{infinitely often}
\end{equation}
and 
\begin{equation} \label{th2_d}
\delta_{\min}^{(\alpha)} (N) \geq \frac{(\log N)^{2c+\ve}}{N} \qquad \text{only finitely often},
\end{equation}
for almost all $\alpha > 0$. We have not been able to establish \eqref{th2_a} and \eqref{th2_d}, but we will comment on them after giving the proof of Theorem \ref{th2}.

\section{Proof of Theorem \ref{th_cat}}

We start with the proof of Theorem \ref{th_cat}. The following statement was long known as Catlin's conjecture \cite{cat}. It was recently established as a consequence of Koukoulopoulos and Maynard's proof of the Duffin--Schaeffer conjecture, see \cite[Theorem 2]{km}.

\begin{lemma} \label{lemma_cat}
Let $(\psi(k))_{k \geq 1}$ be a sequence of non-negative reals. Let $\mathcal{B}$ denote the set of those $\alpha \in [0,1]$ for which the inequality
$$
\left| \alpha - \frac{a}{k} \right| 	\leq \frac{\psi(k)}{k}
$$
has infinitely many solutions $(a,k)$ with $0 \leq a \leq k$. Then $\lambda(\mathcal{B}) = 0$ or $\lambda(\mathcal{B})=1$, according to whether the series 
$$
\sum_{k=1}^\infty \varphi(k) \sup_{b \geq 1} \left\{\frac{\psi(b k)}{bk} \right\}.
$$
is convergent or divergent, respectively.
\end{lemma}

Now we give the proof of Theorem \ref{th_cat}. Let $(\eta(k))_{k \geq 1}$ be a sequence of non-negative reals. Let $A = \bigcup_{N \geq 1} (A_N - A_N)^+$. First we settle the case when $\limsup_{k \to \infty} \eta(k) >0$. It can be shown quite easily that in this case $\|(a_m-a_n)\alpha\| \leq \eta(N)$ has infinitely many solutions for almost all $\alpha$. This follows for example from Weyl's general equidistribution result \cite[Satz 21]{weyl}, together with the simple observation that $A$ is an infinite set of positive integers. Thus in this case we have $\lambda(\mathcal{A}) = 1$. It is also not difficult to see that the series \eqref{con_div} is necessarily divergent in this case. Fix a number $k \in A$. Then we have $\mathcal{N}(k) < \infty$, where $\mathcal{N}(k)$ is defined as in \eqref{nkdef}. Taking a sum over all divisors $d$ of $k$, we have
$$
\sum_{ \substack{d | k,\\ d \geq \log k}} \varphi(d)  \sup_{b \geq 1} \left\{ \frac{\sup_{\ell \geq \mathcal{N}(bd)} \eta(\ell)}{bd} \right\} \geq \sum_{ \substack{d | k,\\ d \geq \log k}} \varphi(d)  \frac{\sup_{\ell \geq \mathcal{N}(k)}\eta(\ell)}{k} \geq \frac{ \sup_{\ell \geq \mathcal{N}(k)}\eta(\ell)}{2}
$$
for sufficiently large $k$, where due to $\mathcal{N}(k)<\infty$ the supremum is non-empty. By assumption the supremum is bounded below by an absolute constant, and since we can find arbitrarily large $k \in A$ this can be used to deduce that \eqref{con_div} is indeed divergent in this case.

For the rest of this proof we can assume that $\lim_{k \to \infty} \eta(k) = 0$. Set
$$
\psi(k) = \sup_{\ell \geq \mathcal{N}(k)} \eta(\ell) \qquad \textrm{and} \qquad \psi^{*}(k) = \varphi(k) \cdot \sup_{b \geq 1} \left\{\frac{\psi(b k)}{bk} \right\}.
$$
For every $k \in A$ we set
\begin{equation} \label{Sk_def}
S_k = [0,1] \cap \left( \bigcup_{0 \leq a \leq k} \left( \frac{a}{k} - \frac{\psi(k)}{k}, \frac{a}{k} + \frac{\psi(k)}{k} \right) \right).
\end{equation}
For convenience we also set $S_k = \emptyset$ for all $k \not\in A$. Let $\mathcal{B}$ denote the set of those $\alpha \in [0,1]$ which are contained in $S_k$ for infinitely many values of $k$. By Lemma \ref{lemma_cat} we have $\lambda(\mathcal{B})=0$ or $\lambda(\mathcal{B})=1$, according to whether the series
$$
\sum_{k=1}^\infty \varphi(k) \sup_{b \geq 1} \left\{\frac{\psi(b k)}{bk} \right\} = \sum_{k=1}^\infty \varphi(k) \sup_{b \geq 1} \left\{\frac{\sup_{\ell \geq \mathcal{N}(bk)} \eta(\ell)}{bk} \right\}
$$
converges or diverges, respectively. Note that this is the same series as in equation \eqref{con_div} in the statement of Theorem \ref{th_cat}. To complete the proof of Theorem \ref{th_cat}, we will show that the set $\mathcal{A}$ defined in the statement of the theorem is actually the same as the set $\mathcal{B}$ defined above.

Let us first assume that $\alpha \in \mathcal{A}$. Then there are infinitely many values of $N,m,n$ with $m \neq n$ and $1 \leq m,n \leq N$ such that $\|(a_m-a_n) \alpha \| \leq \eta(N)$. Set $k = a_m - a_n$. Since $m,n \leq N$ we have $\mathcal{N}(k) \leq N$. Thus we have $\psi(k) \geq \eta(N)$. Consequently $\|(a_m - a_n) \alpha\| = \|k \alpha\| \leq \eta(N)$ implies that $\|k \alpha\| \leq \psi(k)$, and thus $\alpha \in S_k$. Since we are in the case $\eta(N) \to 0$, a particular difference $k = a_m - a_n$ can only generate finitely many values $N$ (together with $m,n \leq N$) such that $\|(a_m - a_n) \alpha\| \leq \eta(N)$. Thus we can find infinitely many different values of $k$ such that $\alpha \in S_k$, in other words we proved that $\mathcal{A} \subset \mathcal{B}$.

Now let us assume that $\alpha \not\in \mathcal{A}$. Then there are only finitely many $N$ and $m,n$ with $m \neq n$ and $1 \leq m,n \leq N$ such that $\|(a_m - a_n) \alpha\| \leq \eta(N)$. This implies that there are only finitely many $k \geq 1$ and $N$ such that $k \in A_N - A_N$ and $\|k \alpha\| \leq \eta(N)$. Since $\eta(N) \to 0$, there are only finitely many $k \geq 1$, each together with finitely many $\ell \geq \mathcal{N}(k)$, such that $\|k \alpha\| \leq \eta(\ell)$. Consequently there are only finitely many $k$ such that $\|k \alpha\| \leq \sup_{\ell \geq \mathcal{N}(k)} \eta(\ell)$, which is the same as saying that there are only finitely many $k$ such that $\alpha \in S_k$.

Thus we have established that $\mathcal{A} = \mathcal{B}$. Since the convergence, respectively the divergence of \eqref{con_div} gives a criterion for a zero-one law for $\lambda(\mathcal{B})$, it consequently also gives a criterion for a zero-one law for $\lambda(\mathcal{A})$, as claimed in the statement of Theorem \ref{th_cat}.

\section{Proof of Theorem \ref{th1}, part 1: Lower bound}\label{lowerbound}

Now we come to the proof of Theorem \ref{th1}. For convenience of writing, we arrange the elements of $A = \bigcup_{N \geq 1} (A_N - A_N)^+$ into a sequence. Thus, let $(z_n)_{n \geq 1}$ be a sequence of distinct integers, such that for all $N$ we have 
\begin{equation} \label{zndef}
\{z_n:~1 \leq n \leq C_N \} = (A_N - A_N)^+.
\end{equation}
The sequence $(z_n)_{n \geq 1}$ is not uniquely defined, since $C_{N+1} - C_N$ can be as large as $N$, but it is not important which possible version of $(z_n)_{n \geq 1}$ we take as long as \eqref{zndef} is satisfied. For $n \ge 1$, we set 
\[
\psi(n) = \frac{1}{n (\log{\sqrt{n}}) (\log_2{ \sqrt{n}})^{1+\ve}}
\]
and
\begin{equation} \label{Sk_def_new}
S_n =[0,1] \cap \left( \bigcup_{0 \leq a \leq z_n} \left( \frac{a}{z_n} - \frac{\psi(n)}{z_n},  \frac{a}{z_n} + \frac{\psi(n)}{z_n} \right) \right).
\end{equation}
We clearly have
$$
\sum_{n=1}^\infty \lambda(S_n) = \sum_{n=1}^\infty \min\{2 \psi(n),1\} \leq \sum_{n=1}^\infty \frac{2}{n (\log{\sqrt{n}}) (\log_2{ \sqrt{n}})^{1 + \ve}} < \infty, 
$$
thus by the first Borel--Cantelli lemma almost all $\alpha \in [0,1]$ are contained in finitely many sets $S_n$. If $\alpha \not\in S_n$ for all sufficiently large $n$, then $\|z_n \alpha\| \geq 1 / (n (\log{\sqrt{n}}) (\log_2{ \sqrt{n}})^{1+ \ve})$ for all sufficiently large $n$. Since $\{z_1, \dots, z_{C_N}\} = (A_N - A_N)^+$, this implies that for all sufficiently large $N$ and all $a_m,a_n \in A_N$ we have 
$$
\|(a_m - a_n) \alpha\| \geq 1 / (C_N (\log{\sqrt{C_N}}) (\log_2 \sqrt{C_N})^{1 + \ve}) \geq 1 / (C_N \log N (\log_2 N)^{1+\ve}),
$$ 
where we used that $C_N \leq N^2$. This proves the lower bound in Theorem \ref{th1}.

\section{Proof of Theorem \ref{th1}, part 2: Upper bound}

As usual in metric number theory, the ``divergence'' part is much more difficult than the ``convergence'' part, since the divergence part of the Borel--Cantelli lemma requires some form of stochastic independence, while the convergence part holds unconditionally; see \cite{bvbc} for a detailed discussion of this issue. In Section \ref{sub_out}, we will prove the upper bound of Theorem \ref{th1} which depends on the size of $a_N$, and sketch the argument leading to a bound that is independent of the size of $a_N$. In Section \ref{sub_over} we introduce the precise construction for this general upper bound and collect several auxiliary results, and in Section \ref{sub_th} we prove the upper bound in Theorem \ref{th1} which is independent of the size of $a_N$.

\subsection{Outline and heuristics} \label{sub_out}

When defining $S_n$ as in \eqref{Sk_def_new} above, but just replacing $\psi(n)$ by a slightly larger function, the resulting set system cannot be assumed to be sufficiently ``independent'' for a direct application of the second Borel--Cantelli lemma -- this is the message from Duffin and Schaeffer's counterexample \cite{ds} to Khintchine's conjecture. It is known that in metric Diophantine approximation one does not need full stochastic independence, but that it is sufficient to establish ``quasi-independence on average'' (cf.\  \cite{bdv,bvbc,bv} as well as Lemma \ref{chung_erd} below). The key for this is to control the measure of the overlaps $S_m \cap S_n$ for $m \neq n$. In Rudnick's paper \cite{rud} this is done by a direct application of $L^2$ methods, which essentially gives the overlap estimate
\begin{eqnarray}  \label{cont}
\sum_{m,n \leq N} \lambda(S_m \cap S_n) \ll \sum_{m,n \leq N} \sqrt{\lambda(S_m) \lambda(S_n)} \frac{\gcd(z_m,z_n)}{\sqrt{z_m z_n}}.
\end{eqnarray}
The sum on the right-hand side of \eqref{cont} is called a GCD sum, and is known to play an important role in metric Diophantine approximation (see \cite[Chapter 3]{hm}). The optimal upper bound for such sums was recently obtained by de la Bret\`eche and Tenenbaum in \cite{dlBT}. In our setting the bound for the GCD sum gives an extra factor $\exp \left(C \sqrt{\log N \log_3 N} / \sqrt{\log_2 N} \right)$, where $C > 0$ is an appropriate absolute constant. Inserting that in \eqref{cont} gives the upper bound
\begin{equation} \label{cont2}
\sum_{m,n \leq N} \lambda(S_m \cap S_n) \ll  \exp \left(C \sqrt{\log N \log_3 N} / \sqrt{\log_2 N} \right)  \sum_{n \leq N} \lambda(S_n). 
\end{equation}
To ensure the quasi-independence property noted above we need
$$
\sum_{m,n\leq N} \lambda(S_m \cap S_n) \ll \sum_{m,n \leq N} \lambda(S_m) \lambda(S_n) = \left( \sum_{n \leq N} \lambda(S_n) \right)^2,
$$
which by \eqref{cont} and \eqref{cont2} can be reformulated as saying that $\sum_{n \leq N} \lambda(S_n)$ needs to exceed $\exp \left(C \sqrt{\log N \log_3 N} / \sqrt{\log_2 N} \right)$. This explains where the extra factor in Rudnick's theorem comes from (which he writes in the less precise form $N^\ve$). We stress the fact that when following such a direct $L^2$ approach this extra factor is essentially optimal.

\subsubsection{Proof of the first upper bound of Theorem \ref{th1}}

To reduce the size of the overlaps $S_m \cap S_n$, one can replace $S_m$ and $S_n$ by modified sets which preserve most of the measure of the original sets, but remove those parts which are excessively responsible for the overlaps. This is  the strategy which led to the co-prime setup in the Duffin--Schaeffer conjecture. Adapting this idea to our situation consider the reduced sets
\begin{equation} \label{red_sets}
S_n^{\textup{coprime}} = [0,1] \cap \left( \bigcup_{\substack{0 \leq a \leq z_n,\\\gcd(a,z_n) = 1}} \left( \frac{a}{z_n} - \frac{\psi(n)}{z_n}, \frac{a}{z_n} + \frac{\psi(n)}{z_n} \right) \right).
\end{equation}
This reflects the fact that overlaps $S_m \cap S_n$ overwhelmingly come from intervals in $S_m$ and $S_n$ which are centered around points $a/z_m$ and $b/z_n$, respectively, for which $a/z_m = b/z_n$, so that either $\gcd(a,z_m) >1$ or $\gcd(b,z_n)>1$. Applying the Koukoulopoulos--Maynard theorem, we deduce that for almost all $\alpha$ we have  $$ 
 \| z_n  \alpha\| \leq \psi(n) \qquad \text{for infinitely many $n$}.
$$ as long as we can ensure that
$$
\sum_{n=1}^\infty \lambda(S_n^{\textup{coprime}}) = \infty. 
$$
Note that $\lambda(S_n^{\textup{coprime}}) = \lambda(S_n) \varphi(z_n)/z_n$, where $\varphi$ is Euler's totient function. It is known that $\varphi(z_n)/z_n \gg (\log_2 z_n)^{-1}$, so we lose a factor of size up to $\log_2 z_n$. Note that this is a function of $z_n$ (i.e.\ depending on the actual size of the \emph{elements} of the difference set $A_N - A_N$), and not a function of $n$ (i.e.\ the cardinality of the difference set). Trivially the largest element of $A_N - A_N$ is at most $a_N$, so that $z_{C_N} \leq a_N$. Setting $\psi(n)=\frac{\log_2 z_n}{n\log n \log_2 n}$ for all $n\geq 1$ and restating the result in terms of $C_N$, this argument leads to the first upper bound claimed in Theorem \ref{th1}: for almost all $\alpha$ we have
$$
\delta_{\min}^{\alpha} (N) \leq \frac{\log_2 a_N}{C_N \log N \log_2 N} \qquad \text{for infinitely many $N$}.
$$

If we want to obtain a result which is independent of the size of $a_N$, we cannot fully reduce to the co-prime setting and suffer a loss of measure that is quantified in terms of $\varphi(a_N)$. Instead, in contrast to \eqref{red_sets} we will only remove those sub-intervals from $S_n$ which are centered at $a/z_n$ for some $a$ and $z_n$ sharing a \emph{small} joint prime factor; this will ensure that we have a loss of measure which is quantified in terms of $N$ (and not in terms of $a_N$). At the same time we need to ensure that we remove sufficiently many sub-intervals from $S_n$ so that the large overlaps causing the appearance of the GCD sum can be significantly reduced. In this way we keep a large proportion of the measure of $S_n$, but remove most of the problematic overlaps.

\subsection{Overlap estimates and auxiliary lemmas} \label{sub_over}

After these heuristics we come to the actual proof. Let $(z_n)_{n \geq 1}$ be defined as in the previous section. Throughout this section we assume that we only consider $m,n$ in a range $(2^{k/2},2^k]$ for some positive integer $k$ (where for simplicity of writing we assume that $k$ is even). We set 
$$
\psi(n) = \frac{1}{2n+1}, \quad n \geq 1,
$$
and
$$
S_n^* = [0,1] \cap \left( \bigcup_{\substack{0 \leq a \leq z_n,\\ p | (a,z_n)  \Rightarrow p > 4^k}} \left( \frac{a}{z_n} - \frac{\psi(n)}{z_n},  \frac{a}{z_n} + \frac{\psi(n)}{z_n} \right) \right).
$$
It can be easily checked that the upper bound of Theorem \ref{th1} follows if we can show that almost all $\alpha$ are contained in infinitely many sets $S_n^*$. We state several lemmas, and then give the proofs of those for which no reference is given.

\begin{lemma} \label{lemma_skr_size}
For all $n \in (2^{k/2},2^k]$
$$
\lambda(S_n^*) \geq \frac{1}{n} \frac{1}{e^{\gamma} \log (4^k)} (1+o(1))
$$
as $k \to \infty$.
\end{lemma}

\begin{lemma}[Pollington--Vaughan style overlap estimate]\label{lemmapv}
Let $m,n \in (2^{k/2},2^k]$ with $m \neq n$. Set
\begin{equation} \label{ddef}
D(z_m,z_n) = \frac{\max(z_m \psi(n), z_n \psi(m))}{\gcd(z_m, z_n)}.
\end{equation}
Furthermore, when $D(z_m,z_n) \geq 1$, then set
\begin{equation} \label{pdef}
P(z_m,z_n) = \prod_{\substack{p | \frac{z_m z_n}{(z_m, z_n)^2},\\ D(z_m,z_n) < p \leq 4^k}} \left(1 + \frac{1}{p} \right),
\end{equation}
where the product ranges over all primes $p$ in the specified range. When $D(z_m,z_n) <1$, then set $P(z_m,z_n) =0$. Then
\begin{equation} \label{lemmaint}
\lambda(S_m^* \cap S_n^*) \ll \frac{\sqrt{\psi(m)\psi(n)}}{4^k} + P(z_m,z_n) \lambda(S_m^*) \lambda(S_n^*) .
\end{equation}
\end{lemma}

\begin{lemma}[{Chung--Erd\H os inequality; see e.g.\ \cite[Theorem 1.4.3d]{ch}}] \label{chung_erd}
Let $\mathcal{A}_m,~1 \leq m \leq M$, be events in a probability space $(\Omega,\mathcal{F},\mathbb{P})$, such that $\mathbb{P} (\mathcal{A}_m)>0$ for at least one value of $m$. Then
$$
\mathbb{P} \left( \bigcup_{m=1}^M \mathcal{A}_m \right) \geq \frac{\left( \sum_{m=1}^M \mathbb{P}(\mathcal{A}_m) \right)^2}{\sum_{1 \leq m,n \leq M} \mathbb{P} (\mathcal{A}_m \cap \mathcal{A}_n)}.
$$
\end{lemma}

\begin{lemma}[{Cassels's zero-one law \cite{cassels}}] \label{cassels}
Let $(\xi(m))_{m \geq 1}$ be non-negative real numbers. For $m \geq 1$ set $\mathcal{A}_m = \bigcup_{0 \leq a \leq m}\left(\frac{a}{m} - \frac{\xi(m)}{m}, \frac{a}{m} + \frac{\xi(m)}{m} \right)$. Let $\mathcal{A}$ be the set of those $\alpha \in  [0,1]$ which are contained in infinitely many sets $\mathcal{A}_m$. Then the Lebesgue measure of $\mathcal{A}$ is either 0 or 1.
\end{lemma}

Finally, we need a lemma due to Koukoulopoulos and Maynard \cite{km}. This is the key ingredient in their recent proof of the Duffin--Schaeffer conjecture. 

\begin{lemma}[{\cite[Proposition 5.4]{km}}] \label{kmlemma}
Let $(\eta(q))_{q \geq 1}$ be a sequence of real numbers in $[0,1/2]$.  Set $M(q,r) = \max\{q \eta(r), ~r \eta(q)\}$. Assume that there are $X < Y$ such that
\begin{equation} \label{assume_that}
\frac{1}{16} \leq \sum_{X \leq q \leq Y} \frac{\eta(q) \varphi(q)}{q} \leq 1/2.
\end{equation}
For $t \geq 1$, set 
$$
L_t(q,r) = \sum_{\substack{p \mid qr/\gcd(q,r)^2 \\ p \geq t }} \frac{1}{p}
$$
and
$$
\mathcal{E}_t = \left\{ (q,r) \in (\mathbb{Z} \cap [X,Y))^2:~\gcd(q,r) \geq t^{-1} M(q,r), L_t(q,r) \geq 10 \right\}.
$$
Then
\begin{equation} \label{impl_const}
\sum_{(q,r) \in \mathcal{E}_t} \frac{\eta(q) \varphi(q)}{q} \frac{\eta(r) \varphi(r)}{r} \ll \frac{1}{t}.
\end{equation}
\end{lemma}
In \cite{km} the lemma is formulated with the numbers $1$ and $2$ in the lower and upper bound, respectively, in equation \eqref{assume_that}, rather than $1/16$ and $1/2$. However, it is easily seen that the statements are equivalent (only the value of the implied constant in \eqref{impl_const} changes).\\

We need to prove Lemmas \ref{lemma_skr_size} and \ref{lemmapv}.

\begin{proof}[Proof of Lemma \ref{lemma_skr_size}]
Using $\psi(n) = 1 /(2n+1)$, we have 
\begin{eqnarray*}
\lambda (S_n^*) & = & \frac{2}{z_n (2n+1)} \# \left\{ 0 \leq a \leq z_n:~p | (a,z_n)  \Rightarrow p > 4^k \right\} \\
& \geq & \frac{2}{z_n (2n+1)} z_n \prod_{p \leq 4^k} \left( 1 - \frac{1}{p} \right) \\
& = & \frac{1}{n} \frac{1}{e^{\gamma} \log 4^k} (1+o(1)),
\end{eqnarray*}
where we used Mertens's third theorem (see for example \cite[Theorem 429]{hw}) to estimate the product over primes. 
\end{proof}

%\begin{proof}[Proof of Lemma \ref{equal_measures}]
%This is a simple observation which follows from the fact that the Lebesgue measure is invariant under %the operation $\alpha \mapsto n \alpha ~\textup{mod}~1$ for an positive integer $n$.
%\end{proof}

\begin{proof}[Proof of Lemma  \ref{lemmapv}]
The lemma is a variation of the Pollington--Vaughan overlap estimate from \cite{pv}. The Lebesgue measure of $S_m^* \cap S_n^*$ is bounded above by 
$$
\min \left( \frac{\psi(m)}{z_m}, \frac{\psi(n)}{z_n} \right)  \sum_{\substack{1 \leq a \leq z_m,\\p | (a,z_m)  \Rightarrow p > 4^k}} \sum_{\substack{1 \leq b \leq z_n,\\p | (b,z_n)  \Rightarrow p > 4^k}} \mathbbm{1} \left(\left| \frac{a}{z_m} - \frac{b}{z_n} \right| \leq \max \left( \frac{\psi(m)}{z_m}, \frac{\psi(n)}{z_n} \right) \right).
$$
The contribution of those pairs $(a,b)$ for which $a/z_m \neq b/z_n$ can be estimated by following Pollington and Vaughan's proof verbatim, just taking into account the fact that we only sifted out primes of size below $4^k$. Very briefly, based on an application of Brun's sieve one estimates this overlap by
\begin{align}\label{overlap_1}
 &\min \left( \frac{\psi(m)}{z_m}, \frac{\psi(n)}{z_n} \right) \underbrace{\sum_{\substack{1 \leq a \leq z_m,\\p | (a,z_m)  \Rightarrow p > 4^k}} \sum_{\substack{1 \leq b \leq z_n,\\p | (b,z_n)  \Rightarrow p > 4^k}}}_{a/z_m \neq b/z_n} \mathbbm{1} \left(\left| \frac{a}{z_m} - \frac{b}{z_n} \right| \leq 2 \max \left( \frac{\psi(m)}{z_m}, \frac{\psi(n)}{z_n} \right) \right) \nonumber\\
&\ll \min \left( \frac{\psi(m)}{z_m}, \frac{\psi(n)}{z_n} \right) \gcd(z_m,z_n) \sum_{\substack{0 < j \leq 2 D(z_m,z_n),\\p | j ~\Rightarrow ~p > 4^k}} 1 \nonumber\\
 &\ll \min \left( \frac{\psi(m)}{z_m}, \frac{\psi(n)}{z_n} \right) \gcd(z_m,z_n) D(z_m,z_n) P(z_m,z_n) \nonumber\\
  &\ll \frac{\psi(m)}{z_m} \frac{\psi(n)}{z_n}P(z_m,z_n),
\end{align}
where we used the representation
$$
D(z_m,z_n) = \max \left( \frac{\psi(m)}{z_m}, \frac{\psi(n)}{z_n} \right)  \frac{z_m z_n}{\gcd(z_m, z_n)}
$$
to write
$$
\min \left( \frac{\psi(m)}{z_m}, \frac{\psi(n)}{z_n} \right) \gcd(z_m,z_n) D(z_m,z_n) = \frac{\psi(m)}{z_m} \frac{\psi(n)}{z_n}.
$$
In the setting of Pollington and Vaughan (that is, the Duffin--Schaeffer setting) there is no contribution to the overlap $S_m^* \cap S_n^*$ from pairs $(a,b)$ with $a/z_m = b/z_n$, due to the complete co-primality condition. In our setting these overlaps contribute 
\begin{eqnarray}
& \ll & \min \left( \frac{\psi(m)}{z_m}, \frac{\psi(n)}{z_n} \right)  \sum_{\substack{0 \leq a \leq z_m,\\p | (a,z_m)  \Rightarrow p > 4^k}} \sum_{\substack{0 \leq b \leq z_n,\\p | (b,z_n)  \Rightarrow p > 4^k}} \mathbbm{1} \left( \frac{a}{z_m} = \frac{b}{z_n} \right).  \label{line_est}
\end{eqnarray}
We have $a/z_m = b/z_n$ whenever
$$
a = j \frac{z_m}{\gcd(z_m,z_n)} \quad \text{and} \quad b = j \frac{z_n}{\gcd(z_m,z_n)}
$$
for some $j$ in $0 \leq j \leq \gcd(z_m,z_n)$. The divisibility requirements $p | (a,z_m)  \Rightarrow p > 4^k$ and $p | (b,z_n)  \Rightarrow p > 4^k$ imply that for all $z_m$ and $z_n$ for which there exists a $p \leq 4^k$ such that the $p$-adic valuation of $z_m$ is different from that of $z_n$, there is no admissible value of $j$, and the double sum above is empty. Since $z_m$ and $z_n$ need to be different numbers, we can only get a non-vanishing contribution if there is a prime $p > 4^k$ such that the $p$-adic valuation of $z_m$ is different from the one of $z_n$. In this case, we necessarily have $\gcd(z_m,z_n) \leq \min(z_m,z_n) / 4^k$. Thus we can estimate \eqref{line_est} by 
\begin{eqnarray*}
& \ll & \min \left( \frac{\psi(m)}{z_m}, \frac{\psi(n)}{z_n} \right) \frac{\min(z_m,z_n)}{4^k} \ll \frac{\sqrt{\psi(m) \psi(n)}}{\sqrt{z_m z_n}} \frac{\sqrt{z_m z_n}}{4^k} \ll \frac{\sqrt{\psi(m) \psi(n)}}{4^k}.
\end{eqnarray*}
In combination with \eqref{overlap_1}, that proves the lemma. 
\end{proof}

\subsection{Proof of Theorem \ref{th1}, upper bound independent of the size of $a_N$.} \label{sub_th}

Assume that $k$ is fixed and ``large'' (and, for simplicity of writing, that it is even). We decompose all numbers $z_n$ in the form 
$$
z_n = z_n^{\textup{small}} \cdot z_n^{\textup{large}}, \qquad 2^{k/2} < n \leq 2^k,
$$
where $z_n^{\textup{small}}$ has only prime factors of size at most $4^k$, and $z_n^{\textup{large}}$ only has prime factors of size larger than $4^k$. Write 
$$
\{b_1, \dots, b_H\} = \left\{z_n^{\textup{large}}, ~2^{k/2} < n  \leq 2^{k} \right\},
$$
for some appropriate $H$, where we assume that $b_1, \dots, b_H$ are sorted in increasing order. Clearly $H \leq 2^{k} - 2^{k/2}$, but $H$ might actually be smaller since we could have $z_m^{\textup{large}}  =  z_n^{\textup{large}}$ for some $m \neq n$. For every $n \in (2^{k/2},2^k]$ we now define
$$
y_n = z_n^{\textup{small}} p^{(k)}_{h}, 
$$
where $h$ is the uniquely defined index for which $z_n^{\textup{large}} = b_h$, and where $p^{(k)}_1$ denotes the smallest prime exceeding $4^k$, $p^{(k)}_2$ denotes the second-smallest prime exceeding $4^k$, and so on. Then for sufficiently large $k$ we have $p_H^{(k)} \leq k^2 4^k$ by a coarse application of the prime number theorem. The point in the construction of the numbers $(y_n)_{2^{k/2} < n \leq 2^k}$ is that on the one hand in Lemma \ref{lemmapv} we can replace $P(z_m,z_n)$ and $D(z_m,z_n)$ by $P(y_m,y_n)$ and $D(y_m,y_n)$ in the relevant situations, since the small prime factors of $z_m$ and of $z_n$ are the same as those of $y_m$ and $y_n$, respectively; see below for details. On the other hand we have
$$
\prod_{p | y_n} \left( 1 - \frac{1}{p} \right)  \leq \prod_{\substack{p | z_n,\\p \leq 4^k}} \left( 1 - \frac{1}{p} \right)  \leq \left(1 - \frac{1}{4^k} \right)^{-1}\prod_{p | y_n} \left( 1 - \frac{1}{p} \right), 
$$
where we used that by construction the ``small'' prime factors of $z_n$ and $y_n$ coincide, and $y_n$ has one additional prime factor which exceeds $4^k$. Consequently
\begin{equation}\label{relation}
\frac{2 \psi(n) \varphi(y_n)}{y_n} \leq \lambda(S_n^*) \leq \frac{2 \psi(n) \varphi(y_n)}{y_n} \left(1 - \frac{1}{4^k} \right)^{-1}.
\end{equation}
Thus we can control the size of the Euler totient function of $y_n$ (while we cannot control it for $z_n$, which might have many large prime factors). This will allow us a straightforward application of Lemma \ref{kmlemma}. By Lemma \ref{chung_erd} we have 
\begin{equation} \label{quot}
\lambda \left( \bigcup_{2^{k/2} < n \leq 2^k} S_n^* \right) \geq \frac{\left( \sum_{2^{k/2} < n  \leq 2^k} \lambda (S_n^*) \right)^2}{\sum_{2^{k/2} < m,n \leq 2^k} \lambda (S_m^* \cap S_n^*)}.
\end{equation}
By Lemma \ref{lemma_skr_size} we have
\begin{equation} \label{sum_meas1}
\sum_{2^{k/2} < n  \leq 2^k} \lambda (S_n^*) \geq \sum_{2^{k/2} < n  \leq 2^k}  \frac{1}{n} \frac{1}{e^{\gamma} \log 4^k} (1+o(1)) \geq 0.14
\end{equation}
for sufficiently large $k$. On the other hand we can assume without loss of generality that
\begin{equation} \label{sum_meas2}
\sum_{2^{k/2} < n  \leq 2^k} \lambda (S_n^*) \leq 0.99
\end{equation}
(if the sum of measures is even larger, we can just delete some of the sets $S_n^*$). Thus we can control the size of the numerator on the right-hand side of \eqref{quot}. To estimate the denominator of the right-hand side of \eqref{quot}, by Lemma \ref{lemmapv} we have
\begin{equation} \label{comb3}
\lambda(S_m^* \cap S_n^*) \ll \frac{\sqrt{\psi(m)\psi(n)}}{4^k} + P(z_m,z_n) \lambda(S_m^*) \lambda(S_n^*).
\end{equation}
Trivially
\begin{equation} \label{comb1}
\sum_{2^{k/2} < m,n \leq 2^k} \frac{\sqrt{\psi(m)\psi(n)}}{4^k} \ll 1.
\end{equation}
Note that whenever $z_m^{\textup{large}} \neq z_n^{\textup{large}}$, then we have $\gcd(z_m,z_n) \leq 4^{-k} \min\{z_m,z_n\}$ and so 
$$
D(z_m,z_n) = \frac{\max(z_m \psi(n), z_n \psi(m))}{\gcd(z_m,z_n)} \geq \frac{4^k \max(z_m, z_n)}{2^k \min(z_m,z_n)} \geq 2^k,
$$
so that in this case
$$
P(z_m,z_n) \ll \prod_{2^k \leq p \leq 4^k} \left(1 + \frac{1}{p} \right) \ll 1.
$$
On the other hand, if $z_m^{\textup{large}} = z_n^{\textup{large}}$ then it is easily seen that $D(y_m,y_n) = D(z_m,z_n)$. In the next displayed formula, all sums are taken over $m,n$ in the range $2^{k/2} < m,n \leq 2^k$. Using \eqref{relation} we can estimate 
\begin{eqnarray}
& & \sum_{m,n} P(z_m,z_n) \lambda(S_m^*) \lambda(S_n^*) \nonumber\\
& \ll & \sum_{\substack{m,n: \\ z_m^{\textup{large}} \neq z_n^{\textup{large}}}} \lambda(S_m^*) \lambda(S_n^*)   + \sum_{\substack{m,n: \\ z_m^{\textup{large}} = z_n^{\textup{large}}}} \lambda(S_m^*) \lambda(S_n^*) \prod_{\substack{p | \frac{z_m z_n}{(z_m, z_n)^2},\\ D(z_m,z_n) < p \leq 4^k}} \left(1 + \frac{1}{p} \right) \nonumber\\
& \ll & \sum_{m,n} \lambda(S_m^*) \lambda(S_n^*) + \sum_{m,n} \frac{\psi(m) \varphi(y_m)}{y_m} \frac{\psi(n) \varphi(y_n)}{y_n}   \prod_{\substack{p | \frac{y_m y_n}{(y_m, y_n)^2},\\ D(y_m,y_n) < p \leq 4^k}} \left(1 + \frac{1}{p} \right). \label{comb2}
\end{eqnarray}
We apply Lemma \ref{kmlemma} with
$$
\eta(q) = \psi(q) \qquad \text{for all $q \in \left\{y_n:~ 2^{k/2} < n \leq 2^k \right\}$}
$$
and $\eta(q) = 0$ otherwise, and with $X$ and $Y$ defined as the minimum and maximum, respectively, of the set $\left\{y_n:~ 2^{k/2} < n \leq 2^k \right\}$. Then by \eqref{quot}, \eqref{sum_meas1} and \eqref{sum_meas2} we have
\begin{eqnarray} \label{sum_of_meas}
\sum_{X \leq q \leq Y} \frac{\eta(q) \varphi(q)}{q} & = & \sum_{2^{k/2} < n \leq 2^k} \frac{\psi(n) \varphi(y_n)}{y_n} \in [1/16, 1/2]
\end{eqnarray}
for sufficiently large $k$. Thus an application of Lemma \ref{kmlemma} gives \eqref{impl_const}. Translating \eqref{impl_const} into our situation we obtain
\begin{eqnarray*}
& & \sum_{2^{k/2} < m,n \leq 2^k} \frac{\psi(m) \varphi(y_m)}{y_m} \frac{\psi(n) \varphi(y_n)}{y_n}   \prod_{\substack{p | \frac{y_m y_n}{(y_m, y_n)^2},\\D(y_m,y_n) < p \leq 4^k }} \left(1 + \frac{1}{p} \right) \\
& \ll & \sum_{t=0}^{k} \sum_{\substack{2^{k/2} < m,n \leq 2^k,\\D(y_m,y_n) \in [4^t, 4^{t+1})}} \frac{\psi(m) \varphi(y_m)}{y_m} \frac{\psi(n) \varphi(y_n)}{y_n}   \prod_{\substack{p | \frac{y_m y_n}{(y_m, y_n)^2},\\ p > 4^t}} \left(1 + \frac{1}{p} \right) \\
& \ll & \sum_{t=0}^{k} \sum_{\substack{2^{k/2} < m,n \leq 2^k,\\D(y_m,y_n) \leq 4^{t+1}}} \frac{\psi(m) \varphi(y_m)}{y_m} \frac{\psi(n) \varphi(y_n)}{y_n}   \prod_{\substack{p | \frac{y_m y_n}{(y_m, y_n)^2},\\ p > 4^{t+1}}} \left(1 + \frac{1}{p} \right) \\
& \ll & \sum_{t=0}^{k} \sum_{2^{k/2} < m,n \leq 2^k} \frac{\psi(m) \varphi(y_m)}{y_m} \frac{\psi(n) \varphi(y_n)}{y_n}  \# \mathcal{E}_{4^{t+1}} \\
& \ll & \sum_{2^{k/2} < m,n \leq 2^k} \frac{\psi(m) \varphi(y_m)}{y_m} \frac{\psi(n) \varphi(y_n)}{y_n}  \sum_{t=0}^{k}  4^{-t} \\
& \ll & \sum_{2^{k/2} < m,n \leq 2^k} \frac{\psi(m) \varphi(y_m)}{y_m} \frac{\psi(n) \varphi(y_n)}{y_n} \\
& \ll & 1
\end{eqnarray*}
by \eqref{sum_of_meas}. Together with \eqref{comb3}, \eqref{comb1} and \eqref{comb2} this implies
$$
\sum_{2^{n/2} \leq m,n \leq 2^n} \lambda(S_m^* \cap S_n^*) \ll 1.
$$
Thus by \eqref{quot} and \eqref{sum_meas1}, 
$$
\lambda \left( \bigcup_{2^{k/2} < n \leq 2^k} S_n^* \right) \gg 1 
$$
for all sufficiently large $k$, where the implied constant is independent of $k$. Since $k$ can be chosen arbitrarily large, this implies
$$
\lambda \left( \bigcap_{\ell=1}^\infty \left( \bigcup_{n=\ell}^\infty S_n^* \right) \right) > 0,
$$
and since $S_n^* \subset S_n$, we clearly also have
$$
\lambda \left( \bigcap_{\ell=1}^\infty \left( \bigcup_{n=\ell}^\infty S_n \right) \right) > 0.
$$
Thus the measure of the limsup set is positive, which by Cassels's zero-one law (Lemma \ref{cassels}) implies that it actually equals 1. In other words, almost all $\alpha \in [0,1]$ are contained in infinitely many sets $S_n$. Thus for almost all $\alpha$ there are infinitely many $n$ such that
$$
\|z_n \alpha\| \leq \frac{1}{2 n+1}.
$$
This can be rephrased as saying that for almost all $\alpha$ there are infinitely many $m \neq n$ such that $m,n \leq N$ and
$$
\|(a_m - a_n) \alpha \| \leq \frac{1}{2 C_{N-1}+1} \leq \frac{1}{C_{N}},
$$
where for the last inequality we used the general estimate $2 C_{N-1} + 1 \geq C_{N-1} + N - 1 \geq C_N$. This can finally be rewritten as
\begin{equation}
\delta_{\min}^{\alpha} (N) \leq \frac{1}{C_N} \qquad \text{for infinitely many $N$}
\end{equation}
for almost all $\alpha$, as desired.

\subsection{Comments on the proof}
In conclusion, we comment on the problems that arise with the methods used here when attempting to prove an upper bound that holds for all except finitely many $N$, rather than for infinitely many $N$. The key principle in the argument above (as in the proof of the Duffin--Schaeffer conjecture) is to establish quasi-independence on average, which allows one to conclude that the measure of the union set is positive and uniformly bounded away from zero. The step from positive measure to full measure is then taken with the help of the ``abstract'' zero-one law due to Cassels, which relies on ergodic phenomena. This zero-one law is perfectly suited for an ``infinitely many $N$'' result, but cannot be used to deduce an ``all except finitely many $N$'' result. To obtain a result of the latter form, we would need to directly establish that the measure of the union set is not only positive, but actually that the union set has (almost) full measure. Thus instead of ``quasi-independence on average'', where we are allowed to lose constant (uniformly bounded) factors, we would need to establish ``independence on average'' with the correct exact asymptotics instead of the loss of constant factors. This might in principle be doable, but would require establishing suitably adapted versions of all the key tools used during the proof, and in particular would require an adaption of Lemma \ref{kmlemma} and its long and difficult proof.
In the next section, we explain how to modify Rudnick's $L^2$ approach from \cite{rud} to obtain such a result with a weaker upper bound than we expect to be true.

\section{Proof of Theorem \ref{t:allN}} \label{s:pf_allN}
Following Rudnick \cite{rud} we define, for $N \in \mathbb{Z}_{\ge 1}, M > 0$ and $\alpha \in [0, 1]$,
\[
D(N, M)(\alpha) = \sum_{n = 1}^{C_N} \chi_{[- \frac 1{2M}, \frac 1{2M}]}(\alpha z_n),
\] 
where $z_1, \ldots, z_{C_N}$ are the differences in $A_N - A_N$.
Expanding the characteristic function into a Fourier series, we readily obtain that the expected value of $D(N, M)$ is given by
\[ 
\int_0^1 D(N, M)(\alpha) d\alpha = \frac{C_N}{M}.
\]
For the variance, similar arguments as in Rudnick's proof of \cite[Proposition 3]{rud} yield
\[ 
\mathrm{Var}(D(N, M)) \ll \frac 1M \sum_{m,n \le C_N} \frac{\gcd(z_m, z_n)}{\sqrt{z_m z_n}}.
\]
Thanks to the bounds on GCD sums recalled in Section \ref{sub_out}, we obtain that for every $\ve > 0$,
\[
\mathrm{Var}(D(N, M)) \ll \frac{C_N}{M} N^\ve
\]
(where as in Rudnick's paper that $N^\ve$ can be made more precise).

It now follows from the standard argument via Chebyshev's inequality and the Borel--Cantelli lemma (see also \cite[Corollary 5 and its proof]{rud} for details) that we may take $M$ up to $\approx \frac{C_N}{N^\ve}$ along with the observation (see also \cite[Corollary 6 and its proof]{rud}) that this almost surely produces a gap of size at most $\frac 1{2M} \approx \frac {N^\ve}{C_N}$, which implies our claim.

\section{Proof of Theorem \ref{examples}}

 We note that in this specific case we have $C_N \asymp N \log N$, where the symbol $\asymp$ means that $\ll$ as well as $\gg$ hold. The upper bound $C_N \ll N\log N$ is clear from the prime number theorem, while the lower bound $C_N \gg N \log N$ follows for instance from $E_N \ll N^3/\log N$ (see for example \cite[Lemma 4]{HH}) together with \eqref{C_N_E_N}. The ``convergence'' part of Theorem \ref{examples} is immediate from the cardinality estimate of the difference set and the first Borel--Cantelli lemma. For the ``divergence'' part, we argue as follows. We can construct a sequence $(z_n)_{n \geq 1}$ of distinct integers, such that for all $N$ we have 
\begin{equation*}
\{z_n:~1 \leq n \leq C_N \} \subset (A_N - A_N)^+,
\end{equation*}
such that we also have $z_n \gg n$, and such that $(z_n)_{n \geq 1}$ is strictly increasing. All of those properties are made possible by $C_N \asymp N \log N$, which implies that the difference set $(A_N - A_N)^+$ has a relative density within the maximal possible set $(\{1, \dots, p_N\} - \{1, \dots, p_N\})^+$ which is uniformly bounded away from zero.

Thus we have a sequence of ``admissible'' differences $(z_n)_{n \geq 1}$ which has positive lower density within $\mathbb{N}$. Morally speaking, this allows us to further restrict to a sub-sequence for which the Euler totient function is uniformly bounded, and thus avoid any loss in the application of the reduction to the co-prime (Duffin--Schaeffer) setup. More formally, we fix a small constant $\ve>0$, and for $m \geq 1$ we define 
$$
\eta^*(m) = \left\{ \begin{array}{ll} \ve \eta(m) & \text{if $m \in (z_n)_{n \geq 1}$}, \\0 & \text{otherwise,} \end{array} \right.
$$
where $\eta$ comes from the assumptions of the theorem. Then by construction $\eta^*(m)$ is non-increasing on a set of positive lower asymptotic density. We apply \cite[Corollary $3$]{hm} (a result called the Duffin--Schaeffer theorem) and conclude that for almost all $\alpha$ there are infinitely many solutions $m$ of $\|m \alpha\| \leq \eta^*(m)$. Restating the conclusion in terms of our original sequence, this means that for almost all $\alpha$ there are infinitely many $N$ together with $m,n \leq N$ such that
$$
\|(a_m - a_n) \alpha\| \leq \ve \eta(C_N).
$$
The assumptions on $\eta$ together with $C_N \gg N \log N$ ensure that $\eta(C_N) \ll \eta(N \log N)$, so that finally (by choosing $\ve$ sufficiently small) we can conclude that for almost all $\alpha$ there are infinitely many solutions to
$$
\|(a_m - a_n) \alpha\| \leq \eta(N\log N),
$$
as claimed.\\

\begin{rmk}
We note that for this particular example one easily obtains a result of the ``for all except finitely many $N$'' type that is not far from the conjecturally optimal upper bound \eqref{conj_up}. Indeed, the fact that $C_N$ grows very slowly means that the trivial bound $\delta_{\min}^{\alpha}(N) \leq 1/N$ (which holds for all $\alpha$ and all $N$) is only a $\log N$ factor away from \eqref{conj_up}. Since for this example the difference set is particularly dense (a subset of $\mathbb{N}$ of positive density), it is quite possible that methods from metric Diophantine approximation such as those of Schmidt \cite{schmidt} can be adapted to give even stronger results. However, we have not found a result in the literature which could be directly applied to the problem discussed in this section.
\end{rmk}

\section{Proof of Theorem \ref{th_quad}} \label{s:pf_squares}

Since the proof follows the same path as the proof of Theorem \ref{th1}, we only sketch the differences.  The main new ingredient is the observation that the difference set of an initial segment of the squares  is very closely related to the set of integers arising from the Erd\H os multiplication table problem. Indeed, assume that $q=ab$ where $a,b \leq N$ and $a,b$ are of the same parity, then setting $m=\frac{a+b}{2}$ and $n=\frac{a-b}{2}$ we have $q=(m-n)(m+n)=m^2-n^2$ where $m,n \leq N$. It is easy to see that this process can be reversed, and any such difference of squares gives rise to a product of two numbers $a,b$ of the same parity. Consequently it follows from  Ford's asymptotic results on the multiplication table problem \cite[Corollary 3]{ford} that for $a_n=n^2,~n \geq 1$, we have \begin{equation}\label{multsquares} C_N \asymp N^2 (\log N)^{-c} (\log_2 N)^{-3/2}.\end{equation} The lower bound in \eqref{multsquares} follows from a restriction to the elements of the multiplication table $(2a)(2b)$ with $a,b \leq N/2$. The lower bound in Theorem \ref{th_quad} then follows, as usual, from this estimate for the cardinality of the difference set together with the first Borel--Cantelli lemma. 

For the upper bound of Theorem \ref{th_quad} we define the sets $S_n^{\textup{coprime}}$ as in \eqref{red_sets} and apply the Koukoulopoulos--Maynard theorem with 
 $$\psi(n)=\frac{1}{ n \log n \log_2 n} .$$  
 Then for almost all $\alpha$ we have $$ 
 \| z_n  \alpha\| \leq \psi(n) \qquad \text{for infinitely many $n$}
$$ as long as we can ensure that
$$
\sum_n \lambda(S_n^{\textup{coprime}}) = \infty. 
$$
Recall that $\lambda(S_n^{\textup{coprime}}) = \lambda(S_n) \varphi(z_n)/z_n$. The divergence of the series follows from the following simple estimate
\begin{equation}\label{divseries} \sum_{n\leq N} \frac{1}{n\log n \log_2 n} \frac{\varphi(z_n)}{z_n} \gg \sum_{n\leq N} \frac{1}{n\log n \log_2 n \log_3 n}.\end{equation} For the proof of \eqref{divseries} we will utilize \eqref{multsquares} which implies that the difference set of an initial segment of the squares has a relatively large density within the maximal possible set. A lemma due to Koukoulopoulos and Maynard \cite[Lemma $7.3$]{km} states that for all $x,t\ge1$, we have
	\[\	\#\left\{n\le x:~ \sum_{p|n,p\ge t}\frac{1}{p} \ge  1 \right\} \ll x e^{- t},
	\]
where the implied constant is absolute. Choosing $t=2 \log_2 N$ and noting that $\prod_{p \leq \log_2 N} (1-1/p) \leq (1+o(1)) e^{-\gamma} \log_3 N$ by Mertens's third theorem, we deduce that 
\begin{equation*}
\# \left\{n\le N: \varphi(n)/n \leq 2/\log_3 n \right\} \ll \frac{N}{(\log N)^2},
\end{equation*} 
from which \eqref{divseries} follows. Altogether, this implies that for almost all $\alpha$ we have 
$$
\delta_{\min}^{\alpha} (N) \leq \frac{(\log N)^{c-1} (\log_2 N)^{1/2}}{ N^2}  \qquad \text{for infinitely many $N$}.
$$

\begin{rmks}
\begin{enumerate}[(i)] \mbox{}
\item We note that a similar proof is possible if $a_n = n^2$ is replaced by $a_n = a n^2 + b n + c$ for some fixed integers $a,b,c$. The lower bound follows again directly from Ford's results and the first Borel--Cantelli lemma. For the upper bound, one would need a variant of Ford's results under certain (fixed) congruence restrictions (generalizing the parity considerations in the argument above). Such a restricted version of the multiplication table problem is not explicitly addressed in Ford's paper, but can probably be obtained by a simple modification of his methods.

\item We also note that the same fact which gave us good control of the factor coming from the Euler totient function (namely that the difference set is relatively dense within its maximal possible range) also allows us a good control of the GCD sum, and thus better error terms in the $L^2$ method which leads to a result of ``for all except finitely many $N$'' type. To be more specific, it is known that 
\begin{equation} \label{gcd_sum_mn}
\sum_{m,n=1}^N \frac{\gcd(m,n)}{\sqrt{m,n}} \ll N \log N. 
\end{equation}
This follows for example from \cite[Theorem 4.4]{brou} together with a form of the Cauchy--Schwarz inequality (cf.\ also \cite{toth}). Note that this is much better than the worst-case upper bound for a GCD sum for a general set of cardinality $N$, as described in Section \ref{sub_out} above and used in Section \ref{s:pf_allN}. Plugging that estimate for the GCD sum into the argument outlined in Section \ref{s:pf_allN}, one obtains that for almost all $\alpha$,
$$
\delta_{\min}^{\alpha}(N) \leq \frac{(\log N)^{2 + c + \ve}}{C_N}  \qquad \text{for all except finitely many $N$},
$$
or, formulated in terms of $N$ rather than $C_N$,  
$$
\delta_{\min}^{\alpha}(N) \leq \frac{(\log N)^{2 + 2c + \ve}}{N^2}  \qquad \text{for all except finitely many $N$}.
$$
This is better than Regavim's bound $(\log N)^{4 + \ve}/N^2$ in \cite{reg}, but misses the conjecturally optimal bound \eqref{conj_up} by a factor of $(\log N)^{2+c}$. The factor $(\log N)^{2+c}$ could probably be reduced with a bit more effort rather than using \eqref{gcd_sum_mn} right away. However there is a limitation to the possible improvements in \eqref{gcd_sum_mn} as shown in \cite[Theorem $1.1$]{smallGCD}. To conclude, some further ideas would be necessary to reach \eqref{conj_up}.
\end{enumerate}
\end{rmks}

\section{Proof of Theorem \ref{th2}}

We begin with the first conclusion of Theorem \ref{th2}, that is, with equation \eqref{th2_b}. We follow \cite[Proof of Theorem 1.3]{bbrr}. We restrict ourselves to $\alpha \in [1,2]$; for other intervals of length 1 the proof works in exactly the same way. Let $N$ be given, and write $X = X(N)$. We define 
$$
S_{X,N} = \left\{ \alpha \in [1,2]:~\delta_{\min}^{(\alpha)}(N) < 1/X \right\}
$$
and 
$$
S_X^{(q)} = [1,2] \cap \bigcup_{n,n' \leq \sqrt{N}} \left( \frac{n^2 - n'^2}{q} - \frac{1}{qX}, \frac{n^2 - n'^2}{q} + \frac{1}{qX} \right).
$$
We will require a result due to Ford which is more general than the ``multiplication table'' asymptotics which were mentioned in the previous section, and can only be stated (in a special case) after introducing some further notation.  Let $y<z$, and let $H(x,y,z)$ be the number of positive integers of size at most $x$ which have at least one divisor in the range $(y,z]$. Let $u=u(y,z)$ be defined by $z = y^{1+u}$. Then by \cite[Theorem 1 (v)]{ford}, for all sufficiently large $x,y,z$ such that $y \leq \sqrt{x}$ and $2y \leq z \leq y^2$, we have
\begin{equation} \label{ford}
\frac{H(x,y,z)}{x} \asymp u^c (\log 2/u)^{-3/2},
\end{equation}
where $c = 1-(1+\log_2 2)/(\log 2)$. We first use \eqref{ford} to give an upper bound for the measure of $S_X^{(q)}$. We consider only those values of $q$ in the range $[1,N]$ which can be written in the form $q = m^2 - m'^2 = (m+m')(m-m')$ with $m + m' \leq \sqrt{N}$. If $q \leq N/\log N$, then we can use the trivial estimate
\begin{equation} \label{S_X_meas_triv}
\lambda \left(S_X^{(q)}\right) \leq \frac{2}{X}. 
\end{equation}
Now assume that $q \in (N/\log N, N]$. We split the full range $(N/\log N,N]$ into dyadic intervals $(2^r,2^{r+1}]$. Clearly at most $\log_2 N$ values of $r$ are necessary to do so. Assume that $q \in (2^r,2^{r+1}]$. We want to estimate the number of integers in the range $(q,2q]$ which can be written in the form $n^2 - n'^2 = (n+n')(n-n')$ such that $\sqrt{2^r} \leq n + n' \leq 2N$.  The number of such integers is bounded above by $H(2^{r+2},\sqrt{2^r},2\sqrt{N})$, with $H$ as defined above. We can easily verify that the assumptions made before the statement of \eqref{ford} are satisfied. We have
$$
2\sqrt{N} = (\sqrt{2^r})^{1+u}, \quad \text{which implies} \quad u = \frac{\log (4N / 2^r)}{\log 2^r} \ll \frac{\log_2 N}{\log N},
$$
where we used that $N / 2^r \ll \log N$. Thus by \eqref{ford} the number of such integers is bounded above by 
$$
H(2^{r+2},\sqrt{2^r},2\sqrt{N}) \ll 2^r \left(\frac{\log_2 N}{\log N}\right)^c \left(\frac{1}{\log_2 N}\right)^{3/2} \ll 2^r (\log N)^{-c}(\log_2 N)^{c-3/2},
$$
and so 
\begin{equation} \label{S_X_meas}
\lambda(S_X^{(q)}) \ll \frac{2^r(\log_2 N)^{c-3/2}}{q X (\log N)^{c}} \ll \frac{(\log_2 N)^{c-3/2}}{X (\log N)^c}
\end{equation}
since we assumed that $q \in (2^r,2^{r+1}]$.

We only need to consider $S_X^{(q)}$ for values of $q$ which can be written in the form $q = m^2 - m'^2 = (m+m')(m-m')$ with $m+m' \leq 2\sqrt{N}$. Thus we again invoke Ford's estimate (this time in the classical ``multiplication table'' setup), by which the number of such values of $q$ is bounded above by $\ll N (\log N)^{-c} (\log_2 N)^{-3/2}$. We have
$$
S_{X,N} \subset \bigcup_{q = m^2 - m'^2} S_X^{(q)},
$$
where the union is over all $q = m^2-m'^2$ for which $m+m' \leq 2 \sqrt{N}$. In the next displayed formula we understand that all summations are only taken over integers $q \leq N$ which admit a representation $q = m^2 - m'^2$ as described above. Then by \eqref{S_X_meas_triv} and \eqref{S_X_meas}, we have
\begin{eqnarray*}
\lambda (S_{X,N}) & \ll & \sum_{q} \lambda(S_X^{(q)}) \\
& \ll & \sum_{q \leq N/\log N} \lambda(S_X^{(q)}) + {\sum_{N/\log N < q \leq N}}  \lambda (S_X^{(q)}) \\
& \ll & \frac{N}{\log N} \frac{1}{X} + \frac{N}{(\log N)^{c} (\log_2 N)^{3/2} } \frac{(\log_2 N)^{c-3/2}}{X (\log N)^c} \\
& \ll & \frac{N}{X (\log N)^{2c} (\log_2 N)^{3-c}}.
\end{eqnarray*}
Setting $X = X(N) = N (\log N)^{1-2c}(\log_2 N)^{c-2+\varepsilon}$, we have
$$
\lambda (S_{X,N}) \ll (\log N)^{-1} (\log_2 N)^{-1-\varepsilon}. 
$$
Thus, letting $N$ run along the sequence $2^w$ of powers of $2$, we have
$$
\sum_{w=1}^\infty \lambda (S_{X,2^w}) < \infty,
$$
which implies by the Borel--Cantelli lemma that almost surely only finitely many events $S_{X,2^w}$ happen. In other words, for almost all $\alpha$ 
$$
\delta_{\min}^{(\alpha)}(2^w) > 1/(2^w (\log 2^w)^{1-2c}(\log_2 2^w)^{c-2+\varepsilon})
$$
for all sufficiently large $w$. Clearly $\delta_{\min}^{(\alpha)}(N) \geq \delta_{\min}^{(\alpha)}(2^w)$ for all $N \in [2^{w-1},2^w]$. Also, $N (\log N)^{1-2c} (\log_2 N)^{c-2+\varepsilon}\geq   \frac{2^w}{3} (\log 2^w)^{1-2c}(\log_2 2^w)^{c-2+\varepsilon}$ for  all $N \in [2^{w-1},2^w]$ for sufficiently large $w$. Thus for almost all $\alpha$ 
$$
\delta_{\min}^{(\alpha)}(N) > \frac{(\log N)^{2c}(\log_2 N)^{2-c-\varepsilon}}{3 N \log N} \qquad \text{for all} \quad N \geq N_0(\alpha),
$$
which proves \eqref{th2_b} in a slightly stronger form (note that the exponent of the $\log_2 N$ term is positive). 

Equation \eqref{th2_c} can be established along similar lines. We choose $X(N) = \frac{N(\log_2 N)^{c-3+\varepsilon}}{ (\log N)^{2c}}$. Then by \eqref{S_X_meas_triv} and \eqref{S_X_meas} 
$$
\lambda(S_{X,N}) \ll \frac{N}{\log N} \frac{1}{X} + \frac{N}{X (\log N)^{2c} (\log_2 N)^{3-c}} \ll \frac{1}{(\log_2 N)^{\varepsilon}},
$$
so that $\lambda(S_{X,N}) \to 0$ as $N \to \infty$. Now let 
$$
S:= \left\{ \alpha \in [1,2]:~\delta_{\min}^{(\alpha)}(N) \leq \frac{(\log N)^{2c}(\log_2 N)^{3-c-\varepsilon}}{N} \text{ for all sufficiently large } N \right\}.
$$ Clearly $$S= \liminf_{N \to \infty} S_{X(N),N}$$ and thus $\lambda(S)=0$. This proves \eqref{th2_c}, again in a slightly stronger form. \\

\begin{rmk}
Finally, we note why \eqref{th2_a} and \eqref{th2_d} are more difficult to establish. As usual in metric number theory, the ``divergence'' part based on an application of the second Borel--Cantelli lemma is much more delicate, as it requires some kind of stochastic independence. In contrast to the setup of Theorem \ref{th1}, in Theorem \ref{th2} we are dealing with metric Diophantine approximation where both the numerators and denominators are restricted to coming from a special set. More precisely, we are essentially dealing with the sets
$$
S_X^{(q)} = [1,2] \cap \bigcup_{a \in \mathcal{M}} \left( \frac{a}{q} - \frac{1
}{q X}, \frac{a}{q} + \frac{1}{q X} \right),
$$
where $a$ and $q$ are both restricted to the set $\mathcal{M} = \mathcal{M}(N)$ of integers appearing in the appropriate multiplication table. As we saw in the proof of Lemma \ref{lemmapv}, to handle the overlaps
$$
S_X^{(q)} \cap S_X^{(r)}, 
$$
one is led to counting the number of solutions of a Diophantine inequality
$$
\left| \frac{a}{q} - \frac{b}{r} \right| \leq \frac{\Delta (q,r)}{\textup{lcm}(q,r)}
$$
for some appropriate $\Delta(q,r)$, which is defined by analogy with \eqref{ddef}, under the extra requirement that only solutions $a,b \in \mathcal{M}$ are counted. If the fine arithmetic structure of the set $\mathcal{M}$ is sufficiently ``random'', then one would expect that the number of solutions $(a,b)$ of the equation above has a scaling factor $|\mathcal{M}|^2 / N^2$ which reflects the relative density $\frac{\# \mathcal{M}}{\#\{1, \dots, N\}}$ that arises from the requirement that $a,b \in \mathcal{M}$. If one could establish the fact that the number of solutions (subject to $a,b \in \mathcal{M}$) exhibits the correct scaling (on average, when summing over $q$ and $r$), then this would allow to deduce \eqref{th2_a}. We do not see any reason why the particular set $\mathcal{M}$, arising from the multiplication table problem, should \emph{not} yield the desired scaling factor, but we have not been able to prove anything in this direction. The estimate \eqref{th2_d} seems to be even more delicate, as it does not seem that controlling the pairwise overlaps is sufficient to obtain such a result, but rather the overlaps between more than two sets would need to be controlled. We note in conclusion that there do exist results on metric Diophantine approximation with two restricted variables, most prominently in the work of Harman \cite{har1,har2,har3,har4}. However, his results are not applicable in this setup. His general results \cite{har4} require the admissible set of numerators to have positive upper density (which we do not have in our setting). He also has results without such a density hypothesis, for example for the case when the numerators are restricted to being primes \cite{har3}, but those rely on deep arithmetic properties of the particular set of admissible numerators, and cannot (as far as we can see) be adapted to our setting. 
\end{rmk}

\section*{Acknowledgments}

CA is supported by the Austrian Science Fund (FWF), projects F-5512, I-3466, I-4945, P-34763 and Y-901. DE is supported by FWF projects F-5512, P-34763 and Y-901. MM is supported by FWF project P-33043. The authors are grateful to Shvo Regavim and Zeev Rudnick for several helpful comments.


\begin{thebibliography}{10}

\bibitem{alt}
C.~Aistleitner, T.~Lachmann, and N.~Technau.
\newblock There is no {K}hintchine threshold for metric pair correlations.
\newblock {\em Mathematika}, 65(4):929--949, 2019.

\bibitem{all}
C.~Aistleitner, G.~Larcher, and M.~Lewko.
\newblock Additive energy and the {H}ausdorff dimension of the exceptional set
  in metric pair correlation problems. {W}ith an appendix by {J}ean {B}ourgain.
\newblock {\em Israel J. Math.}, 222(1):463--485, 2017.

\bibitem{Banks}
W.~D. Banks and D.~J. Covert.
\newblock Sums and products with smooth numbers.
\newblock {\em J. Number Theory}, 131(6):985--993, 2011.

\bibitem{bdv}
V.~Beresnevich, D.~Dickinson, and S.~Velani.
\newblock Measure theoretic laws for lim sup sets.
\newblock {\em Mem. Amer. Math. Soc.}, 179(846):x+91, 2006.

\bibitem{bvbc}
V.~Beresnevich and S.~Velani.
\newblock {\em The Divergence {B}orel--{C}antelli Lemma revisited}.
\newblock Preprint. arXiv:2103.12200.

\bibitem{bv}
V.~Beresnevich and S.~Velani.
\newblock Classical metric {D}iophantine approximation revisited: the
  {K}hintchine-{G}roshev theorem.
\newblock {\em Int. Math. Res. Not. IMRN}, (1):69--86, 2010.

\bibitem{bbrr}
V.~Blomer, J.~Bourgain, M.~Radziwi{\l}{\l}, and Z.~Rudnick.
\newblock Small gaps in the spectrum of the rectangular billiard.
\newblock {\em Ann. Sci. \'{E}c. Norm. Sup\'{e}r. (4)}, 50(5):1283--1300, 2017.

\bibitem{bw}
T.~F. Bloom and A.~Walker.
\newblock G{CD} sums and sum-product estimates.
\newblock {\em Israel J. Math.}, 235(1):1--11, 2020.

\bibitem{brou}
K.~A. Broughan.
\newblock The gcd-sum function.
\newblock {\em J. Integer Seq.}, 4(2):Article 01.2.2, 19, 2001.

\bibitem{Carmon}
D.~Carmon.
\newblock Evenly divisible rational approximations of quadratic
  irrationalities.
\newblock {\em Israel J. Math.}, 223(1):441--448, 2018.

\bibitem{cassels}
J.~W.~S. Cassels.
\newblock Some metrical theorems in {D}iophantine approximation. {I}.
\newblock {\em Proc. Cambridge Philos. Soc.}, 46:209--218, 1950.

\bibitem{cat}
P.~A. Catlin.
\newblock Two problems in metric {D}iophantine approximation. {I} \& {II}.
\newblock {\em J. Number Theory}, 8(3):282--288 \& 289--297, 1976.

\bibitem{ch}
T.~K. Chandra.
\newblock {\em The {B}orel-{C}antelli lemma}.
\newblock SpringerBriefs in Statistics. Springer, Heidelberg, 2012.

\bibitem{smallGCD}
R.~de~la Bret\`eche, M.~Munsch, and G.~Tenenbaum.
\newblock Small {G}\'{a}l sums and applications.
\newblock {\em J. Lond. Math. Soc. (2)}, 103(1):336--352, 2021.

\bibitem{dlBT}
R.~de~la Bret\`eche and G.~Tenenbaum.
\newblock Sommes de {G}\'{a}l et applications.
\newblock {\em Proc. Lond. Math. Soc. (3)}, 119(1):104--134, 2019.

\bibitem{dev}
L.~Devroye.
\newblock Upper and lower class sequences for minimal uniform spacings.
\newblock {\em Z. Wahrsch. Verw. Gebiete}, 61(2):237--254, 1982.

\bibitem{dts}
M.~Drmota and R.~F. Tichy.
\newblock {\em Sequences, discrepancies and applications}, volume 1651 of {\em
  Lecture Notes in Mathematics}.
\newblock Springer-Verlag, Berlin, 1997.

\bibitem{ds}
R.~J. Duffin and A.~C. Schaeffer.
\newblock Khintchine's problem in metric {D}iophantine approximation.
\newblock {\em Duke Math. J.}, 8:243--255, 1941.

\bibitem{emm}
A.~Eskin, G.~Margulis, and S.~Mozes.
\newblock Quadratic forms of signature {$(2,2)$} and eigenvalue spacings on
  rectangular 2-tori.
\newblock {\em Ann. of Math. (2)}, 161(2):679--725, 2005.

\bibitem{ford}
K.~Ford.
\newblock The distribution of integers with a divisor in a given interval.
\newblock {\em Ann. of Math. (2)}, 168(2):367--433, 2008.

\bibitem{hw}
G.~H. Hardy and E.~M. Wright.
\newblock {\em An introduction to the theory of numbers}.
\newblock Oxford University Press, Oxford, sixth edition, 2008.

\bibitem{har1}
G.~Harman.
\newblock Metric {D}iophantine approximation with two restricted variables.
  {I}. {T}wo square-free integers, or integers in arithmetic progressions.
\newblock {\em Math. Proc. Cambridge Philos. Soc.}, 103(2):197--206, 1988.

\bibitem{har2}
G.~Harman.
\newblock Metric {D}iophantine approximation with two restricted variables.
  {II}. {A} prime and a square-free integer.
\newblock {\em Mathematika}, 35(1):59--68, 1988.

\bibitem{har3}
G.~Harman.
\newblock Metric {D}iophantine approximation with two restricted variables.
  {III}. {T}wo prime numbers.
\newblock {\em J. Number Theory}, 29(3):364--375, 1988.

\bibitem{har4}
G.~Harman.
\newblock Metric {D}iophantine approximation with two restricted variables.
  {IV}. {M}iscellaneous results.
\newblock {\em Acta Arith.}, 53(2):207--216, 1989.

\bibitem{hm}
G.~Harman.
\newblock {\em Metric number theory}, volume~18 of {\em London Mathematical
  Society Monographs. New Series}.
\newblock The Clarendon Press, Oxford University Press, New York, 1998.

\bibitem{HH}
N.~Hegyv\'{a}ri and F.~Hennecart.
\newblock On monochromatic sums of squares and primes.
\newblock {\em J. Number Theory}, 124(2):314--324, 2007.

\bibitem{km}
D.~Koukoulopoulos and J.~Maynard.
\newblock On the {D}uffin-{S}chaeffer conjecture.
\newblock {\em Ann. of Math. (2)}, 192(1):251--307, 2020.

\bibitem{kn}
L.~Kuipers and H.~Niederreiter.
\newblock {\em Uniform distribution of sequences}.
\newblock Wiley-Interscience, 1974.

\bibitem{mark}
J.~Marklof.
\newblock The {B}erry-{T}abor conjecture.
\newblock In {\em European {C}ongress of {M}athematics, {V}ol. {II}
  ({B}arcelona, 2000)}, volume 202 of {\em Progr. Math.}, pages 421--427.
  Birkh\"{a}user, Basel, 2001.

\bibitem{mato}
K.~Matom\"{a}ki.
\newblock The distribution of {$\alpha p$} modulo one.
\newblock {\em Math. Proc. Cambridge Philos. Soc.}, 147(2):267--283, 2009.

\bibitem{pv}
A.~D. Pollington and R.~C. Vaughan.
\newblock The {$k$}-dimensional {D}uffin and {S}chaeffer conjecture.
\newblock {\em Mathematika}, 37(2):190--200, 1990.

\bibitem{reg}
S.~Regavim.
\newblock {\em Minimal gaps and additive energy in real-valued sequences}.
\newblock Preprint. arXiv:2106.04261.

\bibitem{rud}
Z.~Rudnick.
\newblock A metric theory of minimal gaps.
\newblock {\em Mathematika}, 64(3):628--636, 2018.

\bibitem{rs}
Z.~Rudnick and P.~Sarnak.
\newblock The pair correlation function of fractional parts of polynomials.
\newblock {\em Comm. Math. Phys.}, 194(1):61--70, 1998.

\bibitem{rsz}
Z.~Rudnick, P.~Sarnak, and A.~Zaharescu.
\newblock The distribution of spacings between the fractional parts of
  {$n^2\alpha$}.
\newblock {\em Invent. Math.}, 145(1):37--57, 2001.

\bibitem{RZ1999}
Z.~Rudnick and A.~Zaharescu.
\newblock A metric result on the pair correlation of fractional parts of
  sequences.
\newblock {\em Acta Arith.}, 89(3):283--293, 1999.

\bibitem{RZ2002}
Z.~Rudnick and A.~Zaharescu.
\newblock The distribution of spacings between fractional parts of lacunary
  sequences.
\newblock {\em Forum Math.}, 14(5):691--712, 2002.

\bibitem{sar}
P.~Sarnak.
\newblock Values at integers of binary quadratic forms.
\newblock In {\em Harmonic analysis and number theory ({M}ontreal, {PQ},
  1996)}, volume~21 of {\em CMS Conf. Proc.}, pages 181--203. Amer. Math. Soc.,
  Providence, RI, 1997.

\bibitem{schmidt}
W.~Schmidt.
\newblock A metrical theorem in diophantine approximation.
\newblock {\em Canadian J. Math.}, 12:619--631, 1960.

\bibitem{ten}
G.~Tenenbaum.
\newblock Un probl\`eme de probabilit\'{e} conditionnelle en arithm\'{e}tique.
\newblock {\em Acta Arith.}, 49(2):165--187, 1987.

\bibitem{toth}
L.~T\'{o}th.
\newblock A survey of gcd-sum functions.
\newblock {\em J. Integer Seq.}, 13(8):Article 10.8.1, 23, 2010.

\bibitem{Walker}
A.~Walker.
\newblock The primes are not metric {P}oissonian.
\newblock {\em Mathematika}, 64(1):230--236, 2018.

\bibitem{weyl}
H.~Weyl.
\newblock \"{U}ber die {G}leichverteilung von {Z}ahlen mod. {E}ins.
\newblock {\em Math. Ann.}, 77(3):313--352, 1916.

\end{thebibliography}
\end{document}